\documentclass[12pt,a4paper,english,reqno]{amsart}
\usepackage[a4paper,footskip=1.5em]{geometry}
\usepackage{amsmath,amssymb,amsthm,mathtools}
\usepackage[mathscr]{euscript}
\usepackage[usenames,dvipsnames]{color}
\usepackage{adjustbox,tikz,calc,graphics,babel,standalone}
\usepackage{subcaption}
\usepackage{csquotes,enumerate,verbatim}
\usepackage[final]{microtype}
\usepackage[numbers]{natbib}
\usetikzlibrary{shapes.misc,calc,intersections,patterns,decorations.pathreplacing}
\usepackage{hyperref}
\hypersetup{colorlinks=true,linkcolor=blue,citecolor=blue,pdfpagemode=UseNone,pdfstartview={XYZ null null 1.00}}
\usepackage{cmtiup}

\theoremstyle{plain}
\newtheorem*{theorem*}{Theorem}
\newtheorem{theorem}{Theorem}[section]

\newtheorem*{claim*}{Claim}

\newtheorem{conjecture}[theorem]{Conjecture}

\theoremstyle{remark}

\def\N{\mathbb{N}}
\def\Z{\mathbb{Z}}

\def\P{\mathbb{P}}

\def\B{\mathbf}

\let\originalleft\left
\let\originalright\right
\renewcommand{\left}{\mathopen{}\mathclose\bgroup\originalleft}
\renewcommand{\right}{\aftergroup\egroup\originalright}

\makeatletter
\def\imod#1{\allowbreak\mkern10mu({\operator@font mod}\,\,#1)}
\makeatother

\begin{document}

\title{Connections in randomly oriented graphs}

\author{Bhargav Narayanan}
\address{Department of Pure Mathematics and Mathematical Statistics, University of Cambridge, Wilberforce Road, Cambridge CB3\thinspace0WB, UK}
\email{b.p.narayanan@dpmms.cam.ac.uk}

\date{6 December 2015}

\begin{abstract}
Given an undirected graph $G$, let us randomly orient $G$ by tossing independent (possibly biased) coins, one for each edge of $G$. Writing $a\rightarrow b$ for the event that there exists a directed path from a vertex $a$ to a vertex $b$ in such a random orientation, we prove that for any three vertices $s$, $a$ and $b$ of $G$, we have 
\[\P(s\rightarrow a \cap s\rightarrow b) \ge \P(s\rightarrow a) \P(s\rightarrow b).\]
\end{abstract}

\maketitle

\section {Introduction}
A very natural notion of a random directed graph is that of a random orientation of a fixed undirected graph. Random orientations of graphs often exhibit counter-intuitive properties. For example, Alm and Linusson~\citep{Tournament} showed that in a random orientation of any sufficiently large complete graph, the event that there is a directed path from $a$ to $s$ and the event that there is a directed path from $s$ to $b$ are positively correlated for any three distinct vertices $s$, $a$ and $b$; this is surprising since conditioning on the existence of a path from $a$ to $s$ would intuitively suggest that edges are typically `oriented towards $s$', and that it should consequently be harder to walk from $s$ to $b$. Random orientations in general, and the correlations between connection events in particular, have been studied by a number of authors; see, for instance, \citep{McD, GNP, FourVert}.

Given a finite undirected graph $G = (V,E)$ and a collection of probabilities $\B{p} = (p_e)_{e \in E}$, we orient the edges of $G$ independently by tossing a $p_e$-biased coin to decide the orientation of an edge $e \in E$. More formally, given $G = (V,E)$ and $\B{p}$ as above, suppose that $V \subset \N$ and define $\vec{G}(\B{p})$ to be a random orientation of $G$ where an edge $e = \{a,b\}   \in E$ with $a <b$ is oriented from $a$ to $b$ with probability $p_e$ and from $b$ to $a$ otherwise, independently of the other edges. We call $\vec G (\B{p})$ a \emph{$\B{p}$-biased orientation} of $G$ and write $\P_{G,\B{p}}$ for the corresponding probability measure. Note that $\vec{G}(\B{p})$ is an unbiased, uniformly random orientation of $G$ when $p_e = 1/2$ for every $e \in E$.

For a pair of vertices $a$ and $b$ of $G$, let $a \rightarrow b$ denote the \emph{connection event} that there is a directed path from $a$ to $b$ in a random orientation of $G$. Our aim in this short paper is to establish the following correlation inequality.
\begin{theorem}\label{corr-ineq}
Let $G=(V,E)$ be an undirected graph. For any three vertices $s,a,b \in V$ and any collection of probabilities $\B{p} = (p_e)_{e\in E}$, we have
\[\P_{G,\B{p}}(s\rightarrow a \cap s\rightarrow b) \ge \P_{G,\B{p}}(s\rightarrow a) \P_{G,\B{p}}(s\rightarrow b).\]
\end{theorem}

The motivation for considering biased orientations in Theorem~\ref{corr-ineq} comes from our lack of understanding of biased orientations of a number of natural graphs, most important of which is perhaps the square lattice. For $0 \le p \le 1$, let $\vec \Z^2(p)$ denote a random orientation of the square lattice obtained as follows: orient a horizontal edge, independently of the other edges, rightwards with probability $p$ and otherwise leftwards, and similarly, orient a vertical edge, independently of the other edges, upwards with probability $p$ and otherwise downwards. The following conjecture is due to Grimmett~\citep{Grim} and remains wide open.

\begin{conjecture} 
For each $p \neq 1/2$, $\vec \Z^2(p)$ almost surely contains an infinite directed path.
\end{conjecture}

Let us mention that while we state and prove Theorem~\ref{corr-ineq} for finite graphs, the result also holds for any graph on a countably
infinite vertex set (such as the square lattice); indeed, this follows from a standard limiting argument. We also remark that the challenge in establishing Theorem~\ref{corr-ineq} arises entirely from having to deal with genuinely biased orientations. 

Indeed, the main difficulty in working with random orientations is that a connection event $a \rightarrow b$ is not `up-closed' in general. In other words, it is not necessarily true that one can find a `good' orientation for each edge with the property that the event $a \rightarrow b$ is closed under the operation of changing the orientation of an edge from `bad' to `good'. For example, it is clear from Figure~\ref{bad} that connection events in a (large) finite grid are neither closed under the operation of changing the orientation of a horizontal edge to the left, nor closed under the operation of changing the orientation of a horizontal edge to the right.

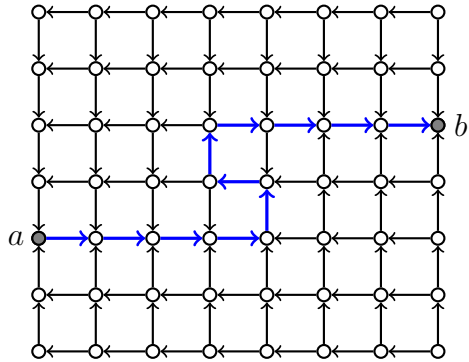
\begin{figure}
	\begin{center}
		\begin{tikzpicture}[scale = 0.75]
		\foreach \x in {1,2,3,4,5,6,7,8}
		\foreach \y in {1,2,3,4,5,6,7}
		\node (l\x\y) at (\x, \y) [inner sep=0.6mm, thick, circle, draw=black!100, fill=black!0] {};
		
		\node at (1, 3) [inner sep=0.6mm, thick, circle, draw=black!100, fill=black!50] {};
		\node at (8, 5) [inner sep=0.6mm, thick, circle, draw=black!100, fill=black!50] {};
		\node at (0.6,3) {$a$};
		\node at (8.4,5) {$b$};
		
		\foreach \x in {1,2,3,4,5,6,7,8}
		\foreach \y[evaluate={\yp=int(\y+1)}] in {5,6}
		\draw[->, thick] (l\x\yp) -- (l\x\y);
		
		\foreach \x in {1,2,3}
		\foreach \y[evaluate={\yp=int(\y+1)}] in {3,4}
		\draw[->, thick] (l\x\yp) -- (l\x\y);
		
		\foreach \x in {1,2,3,4,5,6,7,8}
		\foreach \y[evaluate={\yp=int(\y+1)}] in {1,2}
		\draw[->, thick] (l\x\y) -- (l\x\yp);

		\foreach \x in {6,7,8}
		\foreach \y[evaluate={\yp=int(\y+1)}] in {3,4}
		\draw[->, thick] (l\x\y) -- (l\x\yp);
		
		\foreach \x[evaluate={\xp=int(\x+1)}] in {1,2,3,4,5,6,7}
		\foreach \y in {1,2,6,7}
		\draw[->, thick] (l\xp\y) -- (l\x\y);
		
		\foreach \x[evaluate={\xp=int(\x+1)}] in {1,2,3,5,6,7}
		\foreach \y in {4}
		\draw[->, thick] (l\xp\y) -- (l\x\y);
		
		\foreach \x[evaluate={\xp=int(\x+1)}] in {1,2,3,4}
		\foreach \y in {3}
		\draw[->, very thick, blue] (l\x\y) -- (l\xp\y);

		\foreach \x[evaluate={\xp=int(\x+1)}] in {5,6,7}
		\foreach \y in {3}
		\draw[->, thick] (l\xp\y) -- (l\x\y);
		
		\foreach \x[evaluate={\xp=int(\x+1)}] in {4,5,6,7}
		\foreach \y in {5}
		\draw[->, very thick, blue] (l\x\y) -- (l\xp\y);
		
		\foreach \x[evaluate={\xp=int(\x+1)}] in {1,2,3}
		\foreach \y in {5}
		\draw[->, thick] (l\xp\y) -- (l\x\y);
		
		\draw[->, very thick, blue] (l54) -- (l44);
		\draw[->, very thick, blue] (l53) -- (l54);
		\draw[->, very thick, blue] (l44) -- (l45);
		\draw[->,, thick] (l55) -- (l54);
		\draw[->,, thick] (l44) -- (l43);
		\end{tikzpicture}
	\end{center}
	\caption{Left-to-right connection events in the grid are not `up-closed'.}
	\label{bad}
\end{figure}

This `up-closedness' issue however disappears when we restrict ourselves to unbiased orientations. Indeed, in this case, as was observed by McDiarmid~\citep{McD}, the distribution of the set of vertices reachable from a vertex $s$ in an unbiased orientation of $G$ is identical to the distribution of the connected component of $s$ in the standard percolation model (at density 1/2) on $G$. Therefore, as noted by Linusson~\citep{Linusson}, our result follows instantly from Harris's lemma~\citep{Harris} in this case. However, we see no simple way of deducing Theorem~\ref{corr-ineq} from Harris's lemma in general; instead, our proof relies on the powerful four-functions theorem of Ahlswede and Daykin~\citep{fourfunc}.

The proof of Theorem~\ref{corr-ineq} is given in Section~\ref{proof-sec}. We make a few remarks and conclude this note in Section~\ref{conc}.

\section{Proof of the main result}\label{proof-sec}
To prove Theorem~\ref{corr-ineq}, we shall require the four-functions theorem of Ahlswede and Daykin~\citep{fourfunc}; see~\citep{textbook} for a proof and several related results.

\begin{theorem}\label{4-func}
Let $S$ be a finite set and let $\alpha, \beta, \gamma$ and $\delta$ be functions from the set of all subsets of $S$ to the non-negative reals. If we have 
\[
\alpha(X_1)\beta(X_2) \le \gamma(X_1 \cup X_2)\delta(X_1 \cap X_2)
\]
for any two subsets $X_1, X_2 \subset S$, then 
\[
\sum_{X \subset S}\alpha(X) \sum_{X \subset S}\beta(X) \le \sum_{X \subset S}\gamma(X)\sum_{X \subset S}\delta(X).
\]
\end{theorem}

Before we proceed further, let us introduce some additional notation. For a set of vertices $A$ and a vertex $b$, we write $A \rightarrow b$ for the union of all the events $a \rightarrow b$ with $a \in A$. Theorem~\ref{corr-ineq} is a special case of the following result.

\begin{theorem}\label{corr-set}
Let $G=(V,E)$ be an undirected graph. For any nonempty set $S\subset V$, any pair of vertices $a, b \in V$ and any collection of probabilities $\B{p} = (p_e)_{e\in E}$, we have
\[\P_{G,\B{p}}(S\rightarrow a \cap S\rightarrow b) \ge \P_{G,\B{p}}(S\rightarrow a) \P_{G,\B{p}}(S\rightarrow b).\]
\end{theorem}

\begin{proof}
We prove the theorem by induction on the number of vertices. Clearly, the result holds trivially when $G$ has only one vertex. Therefore, suppose that $G$ has more than one vertex and that we have proved the result for all graphs with fewer vertices than $G$. The inequality is also trivial if either $a \in S$ or $b \in S$, so suppose that neither $a$ nor $b$ belongs to $S$.

Let $H$ denote the graph obtained by deleting $S$ from $G$. Let $T$ denote the set of those vertices of $H$ that are adjacent to some vertex of $S$ in $G$. We write $O_S \subset T$ for the (random) set of those vertices $v \in T$ for which there exists an edge oriented from $S$ to $v$ in $\vec{G}(\B{p})$.

In what follows, to reduce clutter, we write $\P$ for the measure $\P_{G,\B{p}}$ and ${\widehat \P}$ for the measure induced by $\P$ on the graph $H$. For a subset $X \subset T$, let us define
\begin{align*}
\alpha(X) &= \P(O_S = X) {\widehat \P} (X \rightarrow a), \\
\beta(X) &= \P(O_S = X) {\widehat \P} (X \rightarrow b), \\
\gamma(X) &= \P(O_S = X) {\widehat \P} (X \rightarrow a \cap X \rightarrow b), \text{ and}\\
\delta(X) &= \P(O_S = X).
\end{align*}
Note that 
\begin{align*}
\P (S \rightarrow a) = \sum_{X \subset T} \P(O_S = X) \P (S \rightarrow a \,|\, O_S = X) = \sum_{X \subset T} \P(O_S = X) {\widehat \P} (X \rightarrow a),
\end{align*}
so we have
\begin{align*}
\sum_{X \subset T} \alpha(X) &= \P (S \rightarrow a), \\
\sum_{X \subset T} \beta(X) &= \P (S \rightarrow b), \\
\sum_{X \subset T} \gamma(X) &= \P (S \rightarrow a \cap S \rightarrow b), \text{ and}\\
\sum_{X \subset T} \delta(X) &= 1.
\end{align*}

Therefore, by Theorem~\ref{4-func}, to prove our result, it suffices to show that
\[\alpha(X_1)\beta(X_2) \le \gamma(X_1 \cup X_2)\delta(X_1 \cap X_2)
\]
for any two subsets $X_1, X_2 \subset T$. We may inductively assume that we have established Theorem~\ref{corr-set} for $H$. Hence, it follows that
\begin{align*} 
{\widehat \P} (((X_1 \cup X_2) \rightarrow a) \cap ((X_1 \cup X_2) \rightarrow b)) &\ge {\widehat \P} ((X_1 \cup X_2) \rightarrow a) {\widehat \P}( (X_1 \cup X_2) \rightarrow b)\\
&\ge {\widehat \P} (X_1  \rightarrow a) {\widehat \P}( X_2 \rightarrow b).
\end{align*}
Therefore, it suffices to show that 
\[ \P(O_S = X_1) \P(O_S = X_2) \le \P(O_S = X_1 \cup X_2) \P(O_S = X_1 \cap X_2).
\]
This is easy to check. Indeed, each $v \in T$ belongs to $O_S$ with some probability $p_v$, independently of the other vertices of $T$. Hence, we have 
\begin{align*} 
\P(O_S = X_1) \P(O_S = X_2) &= \prod_{v \in X_1 \cap X_2}p_v^2 \prod_{v \in X_1 \triangle X_2}p_v(1-p_v) \prod_{v \notin X_1 \cup X_2}(1-p_v)^2\\
&=\P(O_S = X_1 \cup X_2) \P(O_S = X_1 \cap X_2).
\end{align*}
The conditions of Theorem~\ref{4-func} have been verified; Theorem~\ref{corr-set} now follows by induction.
\end{proof}

\section{Conclusion}\label{conc}
The correlation inequality proved in this paper is `intuitively obvious', and it therefore feels somewhat unsatisfactory that our proof must rely on the four-functions theorem. Finding a more elementary proof of our main result remains an interesting problem.

\bibliographystyle{amsplain}
\bibliography{connections}

\end{document}